\documentclass[12pt]{article}

\usepackage{amsmath,amsthm,amsfonts,latexsym,amscd,amssymb}

\def\qed{{\hbadness=10000\hfill\ \vbox{\hrule height.09ex
     \hbox{\vrule width.09ex height1.55ex depth.2ex \kern1.8ex
     \vrule width.09ex height1.55ex depth.2ex}\hrule height.09ex}\break
     \bigskip}}
\newtheorem{theorem}{Theorem}[section]
\newtheorem{lemma}[theorem]{Lemma}

\newtheorem{proposition}[theorem]{Proposition}
\newtheorem{definition}[theorem]{Definition}

\newtheorem{remark}[theorem]{Remark}
\newtheorem{example}[theorem]{Example}

\begin{document}
\title{On the Congruences in Right Loops}

\author{Vipul Kakkar*
~ and R.P. Shukla**\\
*School of Mathematics, Harish-Chandra Research Institute\\
Allahabad\\
**Department of Mathematics, University of Allahabad \\
Allahabad (India) 211 002\\
Email: vplkakkar@gmail.com; shuklarp@gmail.com}

\date{}
\maketitle

\begin{abstract}
In case the stability relation is a congruence, a necessary and also a sufficient condition for its equality with the center congruence is given. 
\end{abstract}


\noindent \textbf{Keywords:} Right loop, Normalized Right Transversal, Stability Relation, Center Congruence
\\
\noindent \textbf{2000 Mathematics subject classification:} 20D60; 20N05
\section{Introduction}
Let $G$ be a group and $H$ be a subgroup of $G$. A \textit{normalized right transversal (NRT)} $S$ of $H$ in $G$ is a subset of $G$ obtained by choosing one and only one element from each right coset of $H$ in $G$ and $1 \in S$. Then $S$ has an induced binary operation $\circ$ given by $\{x \circ y\}=Hxy \cap S$, with respect to which $S$ is a right loop with identity $1$, that is, a right quasigroup with both sided identity  (see \cite[Proposition 4.3.3, p.102]{smth},\cite{rltr}). Conversely, every right loop can be embedded as an NRT in a group with some universal property (see \cite[Theorem 3.4, p.76]{rltr}). Let $\langle S \rangle$ be the subgroup of $G$ generated by $S$ and $H_S$ be the subgroup $\langle S \rangle \cap H$. Then  $H_S=\langle \left\{xy(x \circ y)^{-1}|x, y \in S \right\} \rangle$ and $H_{S}S=\langle S \rangle$ (see \cite{rltr}).
Identifying $S$ with the set $H \backslash G$ of all right cosets of $H$ in $G$, we get a transitive permutation representation $\chi_{S}:G\rightarrow Sym(S)$ defined by $\left\{\chi_{S}(g)(x)\right\}=Hxg \cap S, g\in G, x\in S$. The kernal $Ker \chi_S$ of this action is $Core_{G}(H)$, the core of $H$ in $G$.   

Let $G_{S}=\chi_{S}(H_{S})$. This group is known as the \textit{group torsion} of the right loop $S$ (see \cite[Definition 3.1, p.75]{rltr}). The group $G_S$ depends only on the right loop structure $\circ$ on $S$ and not on the subgroup $H$. Since $\chi_S$ is injective on $S$ and if we identify $S$ with $\chi_S(S)$, then $\chi_S(\langle S \rangle)=G_SS$ which also depends only on the right loop $S$ and $S$ is an NRT of $G_S$ in $G_SS$. One can also verify that $ Ker(\chi_S|_{H_SS}: H_SS \rightarrow G_SS)=Ker(\chi_S|_{H_S}: H_S \rightarrow G_S)=Core_{H_SS}(H_S)$ and $\chi_S|_S$=the identity map on $S$. Also $(S, \circ)$ is a group if and only if $G_S$ trivial.

A non-empty subset $T$ of right loop $S$ is called a \textit{right subloop} of $S$, if it is right loop with respect to induced binary operation on $T$ (see \cite[Definition 2.1, p. 2683]{rpsc}). Also an \textit{invariant right subloop} of a right loop $S$ is precisely the equivalence class of the identity of a congruence in $S$ (\cite[Definition 2.8, p. 2689]{rpsc}). It is observed in the proof of \cite[Proposition 2.10, p. 2690]{rpsc} that if $T$ is an invariant right subloop of $S$, then the set $S/T=\{T \circ x|x \in S\}$ becomes right loop called as \textit{quotient of S mod H} and the map $\nu:S\rightarrow S/T$ defined by $\nu(x)=T \circ x$ is a right loop homomorphism. It is also observed in this paper that $\nu:S\rightarrow S/T$ induces a group homomorphism $\tilde\nu:G_SS\rightarrow G_{S/T}S/T$ (see the discussion following \cite[Lemma 2.5, p.2684]{rpsc}).


Let $(S,\circ)$ be a right loop. Let $a \in S$ such that the equation $a \circ X=c$ has unique solution for all $c \in S$, in notation we write it as $X=a \backslash_{\circ} c$. Then the map $L_a^{\circ}:S\rightarrow S$ defined by $L_a^{\circ}(x)=a \circ x$ is bijective map. Such an element $a$ is called a \textit{left non-singular} element of $S$. For $x \in S$, we denote the map $y\mapsto y \circ x$ $(y \in S)$ by $R_x^{\circ}$. We will drop the superscript, if the binary operation is clear.

In the sections \ref{6c2s} and \ref{6c3s}, we discuss the central congruence and stability relation $\sigma(S)$ on a right loop $S$. In case $S$ is a loop, $\sigma(S)$ is a congruence and coincides with the center congruence (see \cite[p. 81; Proposition 3.15, p. 83]{smth1}). However, if $S$ is a right loop but not a loop, then $\sigma(S)$ need not be a congruence (see Example \ref{6c3se1}). We have also shown that even if $\sigma(S)$ is a congruence on a right loop, it need not coincide with the central congruence (see Example \ref{6c3se2}). In case, $\sigma(S)$ is a congruence, a necessary and a sufficient condition for its equality with the center congruence is given (Theorem \ref{6c3st1}). In the last section, we have given an example of a right loop which is simple, however the group $G_SS$ is not quasiprimitive on $S$. This example corrects \cite[Theorem 4, p. 474]{phsmth} and shows that it is one directional only.    

\section{Centering Congruence in a Right Loop}\label{6c2s}

We note that a right loop $(S,\circ)$ contains two binary operations $\circ$ and $/$, where $x/y$ is the unique solution of the equation $X \circ y=x$, one nullary operation $e$ defined by $e(x)=1$ for all $x \in S$ and one unary operation $l$ defined by $l(x)=x^{\prime}$, where $1$ denotes the identity of $S$ and $a^{\prime}$ denotes the left inverse of $a$ in the right loop $S$. One can prove that $x/y=x \theta f(y^{\prime},y)^{-1} \circ y^{\prime}$. These operations satisfies following conditions: \[(y/x) \circ x=x,(x \circ y)/y=x,1\circ x=x=x \circ 1,x/x=1,x/1=x.\] One observes that a right loop $S$ is a universal algebra $(S,\circ,/,1)$. Define a ternary operation on $S$ by $P(x,y,z)=(x/y) \circ z$. One can note that $P(x,y,y)=x$ and $P(x,x,y)=y$. The operation $P(x,y,z)$ is called as the \textit{Mal'cev operation} (see \cite[p. 19]{smthmv}).

Let $H$ be an invariant right subloop of a right loop $(S,\circ)$ and $\alpha$ be the congruence on $S$ determined by $H$. One can observe that for $(1 \leq i \leq 3)$, $(x_i,y_i) \in \alpha$ $\Rightarrow (P(x_1,x_2,x_3),P(y_1,y_2,y_3)) \in \alpha$.  Now, we have following proposition:

\begin{proposition}\label{6c2sp1}(\cite[Proposition 4.3.2, p.101]{smth1}) Let $(S,\circ)$ be a right loop. Then
\begin{enumerate}
\item[(i)] The congruences on $S$ are permutable, that is if $\alpha$ and $\beta$ are two congruences on a right loop $S$, then $\alpha \circ \beta=\beta \circ \alpha$.
\item[(ii)] A right subloop $\alpha$ of $S \times S$ is a congruence on $S$ if and only if it contains the diagonal right subloop $\Delta(S)=\{(x,x)|x\in S\}$ of $S \times S$.

\end{enumerate}
\end{proposition}



\begin{remark}\label{6c2sr1} The above proposition is true for more general class called as \textit{Mal'cev Variety} (see \cite{smthmv}). It is observed by J. D. H. Smith in \cite{smthmv} that the notion of congruence behaves good in the Mal'cev variety.
\end{remark}

 \begin{definition}\label{6c2sd1}(\cite[DEFINITIONS 211, p. 24]{smthmv}) Let $\beta$ and $\gamma$ be congruences on a right loop $S$. Let $(\gamma|\beta)$ be a congruence on $\beta$. Then $\gamma$ is said to \textbf{centralize} $\beta$ by means of the \textbf{centering congruence} \index{congruence!centering congruence} $(\gamma|\beta)$ such that following conditions are satisfied:
 
\begin{enumerate}
	\item[(i)] $(x,y) ~(\gamma|\beta) ~(u,v) \Rightarrow x ~\gamma ~u$, for all $(x,y), (u,v) \in \beta$.
	\item[(ii)] For all $(x,y) \in \beta$, the map $\pi : (\gamma|\beta)_{(x,y)} \rightarrow \gamma_x$ defined by $(u, v)\mapsto u$ is a bijection, where for a set $X$ and an equivalence relation $\delta$ on $X$, $\delta_w$ denotes the equivalence class of $w \in X$ under $\delta$. 
	\item[(iii)] For all $(x,y) \in \gamma$, $(x,x) ~(\gamma|\beta) ~(y,y)$.
	\item[(iv)] $(x,y) ~(\gamma|\beta) ~(u,v)\Rightarrow (y,x) ~(\gamma|\beta) ~(v,u) $, for all $(x,y), (u,v) \in \beta$.
	\item[(v)] $(x,y) ~(\gamma|\beta) ~(u,v)$ and $(y,z) ~(\gamma|\beta) ~(v,w)$ $\Rightarrow (x,z) ~(\gamma|\beta) ~(u,w)$, for all $(x,y), (u,v), (y,z)$ and $(v,w)$ in $\beta$.
\end{enumerate}
  
 \end{definition}
 By $(i)$ and $(iv)$, we observe that $(x,y) ~(\gamma|\beta) ~(u,v) \Rightarrow y ~\gamma ~v$.
 
\noindent An equivalent condition for centralizing congruence is obtained in \cite[COROLLARY 224, p. 37]{smthmv}, which is given as follows:
\begin{proposition}(\cite[COROLLARY 224, p. 37]{smthmv})\label{6c2sp2} Let $\beta$ and $\gamma$ be congruences on a right loop $S$. Then $\gamma$ centralizes $\beta$ by means of $(\gamma|\beta)$ if and only if the following two conditions are satisfied
\begin{enumerate}
	\item[(i)] $(x,y) ~(\gamma|\beta) ~(u,v) \Rightarrow x ~\gamma ~u$, for all $(x,y), (u,v) \in \beta$.
   \item[(ii)] For all $x \in S$, $(x,x) ~(\gamma|\beta) ~(x,y) \Rightarrow x= y$
\end{enumerate}   
\end{proposition}
Following remark analyzes the centralizing congruence in groups: 
\begin{remark}\label{6c2sr2} Let $G$ be a group. Let $H$ and $K$ be normal subgroups in $G$. Suppose that $K$ centralizes $H$. Then $hk=kh$ for all $k \in K$ and $h \in H$. Let $\beta=\{(x,y) \in G \times G|yx^{-1} \in H\}$ and $\gamma=\{(x,y) \in G \times G|yx^{-1} \in K\}$ be congruences determined by $H$ and $K$ respectively. Define a relation $(\gamma|\beta)$ on $\beta$ by $(x,y) ~(\gamma|\beta) ~(u,v)\Leftrightarrow $ there exists $h \in H$ and $k \in K$ such that $yx^{-1}=vu^{-1}=h$ and $x^{-1}u=k$. One can check that $(\gamma|\beta)$ is a congruence on $\beta$ and $\gamma$ centralizes $\beta$ by means of $(\gamma|\beta)$. One can also see that $(\gamma|\beta)_{(1,1)}=\{(w,w) \in \beta|w \in K\}$. If $K=G$, then $(\gamma|\beta)_{(1,1)}=\Delta(G)$. 
\end{remark} 
Following holds for right loops:
\begin{proposition}(\cite[PROPOSITION 221, p. 34]{smthmv})\label{6c2sp3} Let $S$ be a right loop. Let $\beta$ and $\gamma$ be congruences on $S$ and let $\gamma$ centralizes $\beta$ by means of $(\gamma|\beta)_1$ and $(\gamma|\beta)_2$. Then $(\gamma|\beta)_1=(\gamma|\beta)_2$.
\end{proposition}
\begin{proposition}(\cite[PROPOSITION 226, p. 38]{smthmv})\label{6c2sp4} Let $\gamma$, $\beta_1$ and $\beta_2$ be congruences on a right loop $S$. If $\gamma$ centralizes $\beta_1$ and $\beta_2$, then $\gamma$ centralizes $\beta_1 \circ \beta_2$.
\end{proposition}

\section{Center of a Right Loop and Stability Relation}\label{6c3s}
Let $S$ be a right loop. If a congruence $\alpha$ on $S$ is centralized by $S \times S$, then it is called a \textbf{central congruence} \index{congruence!central congruence} (see \cite[p. 42]{smthmv}). By Propositions \ref{6c2sp3} and \ref{6c2sp4}, there exists a unique maximal central congruence $\zeta(S)$ on $S$, called as the \textbf{center congruence} of $S$. For a finite right loop, it is product of all centralizing congruences. The \textbf{center} \index{right loop!center} $\mathcal{Z}(S)$ of $S$ is defined as $\zeta_1$, the equivalence class of the identity $1$. 

The group $G_SS$ has the natural action $\star$ on $S$. Consider the set $\sigma(S)=\{(x,y) \in S \times S|stab(G_SS,x)=stab(G_SS,y)\}$, where for a pemutation group $G$ on a set $X$, $stab(G,u)$ denotes the stabilizer of $u \in X$ in $G$. One can check that $\sigma(S)$ is an equivalence relation on $S$. This relation is called as \textbf{stability relation} on $S$. One also observes that if $(x,y) \in \sigma(S)$, then $stab(G_SS,x \star p)=p~ stab(G_SS,x)~ p^{-1}=p~ stab(G_SS,y)~ p^{-1}=stab(G_SS,y \star p)$ for all $p \in G_SS$. Consider the equivalence class $\sigma(S)_1$ of $1 \in S$ of a right loop $S$. Let $x,y,z \in S$. Write equation $(C6)$ of \cite[Definition 2.1, p. 71]{rltr} as $f(y,z)(x)=(R_yR_zR_{y \circ z}^{-1})(x)$. Recall that our convention for the product in the symmetric group $Sym(S)$ is given as $(rs)(x)=s(r(x))$ for $r,s \in Sym(S)$ and $x \in S$. Which means that $G_S=\langle R_yR_zR_{y \circ z}^{-1} |y,z \in S\rangle$. Let $x \in \sigma(S)_1$. Then  $stab(G_SS,x)=G_S$. This means that $R_yR_zR_{y \circ z}^{-1}(x)=x$ for all $y,z \in S$. This implies that $\sigma(S)_1=\{x \in S|x \circ (y \circ z)=(x \circ y)\circ z, \text{for~all~} y,z \in S\}=N_{G_SS}(G_S) \cap S$, where $N_{G}(H)$ denotes the normalizer of $H$ in $G$. One can observe that $\sigma(S)_1$ is a right subloop of $S$ which is indeed a group. In the following proposition, we obtain that the elements of $\sigma(S)_1$ are left non-singular. 

\begin{proposition}\label{6c3sp2} Let $(S,\circ)$ be a right loop with identity $1$. Let $\sigma(S)_1$ be the equivalence class of $1$ under the equivalence relation $\sigma(S)$. Then the elements of $\sigma(S)_1$ are left non-singular. 
\end{proposition}
\begin{proof} Let $x \in \sigma(S)_1$ and $y \in S$. Consider the equation $x \circ X=y$. Since $\sigma(S)_1$ is a right subloop, $x^{\prime} \in \sigma(S)_1$ where $x^{\prime}$ is the left inverse of $x$. Since with respect to the induced operation of $S$, $\sigma(S)_1$ is a group, $x \circ x^{\prime}=1$. Also since for each $u \in \sigma(S)_1$ and $v,w \in S$, $u \circ (v \circ w)=(u \circ v) \circ w$, $x^{\prime}\circ y \in S$ is a solution of $x \circ X=y$.   
\end{proof}

The equivalence relation $\sigma(S)$ is a congruence on a loop (see \cite[p. 81]{smth1}) but it need not be a congruence on a right loop $S$. This is shown in following example whose idea came from a GAP (\cite{gap}) calculation:

\begin{example}\label{6c3se1} Consider a subgroup $G=\langle\{u=(4,5,6),v=(1,2,3),w=(1,4)(2,5,3,6)\} \rangle$ $(\cong (\mathbb{Z}_3 \times \mathbb{Z}_3)\rtimes \mathbb{Z}_4)$ of $Sym(6)$. Let $N=\langle w \rangle$ and $H=\langle  w^2 \rangle$. Let $L=\langle u,v \rangle$ and $T=\{I,w\}$. Then $L \in \mathcal{T}(G,N)$ and $T \in \mathcal{T}(N,H)$. Let $S=(TL \setminus \{u\}) \cup \{w^2 u\}$. Then $S \in \mathcal{T}(G,H)$. One observes that $G_SS \cong G$, $G_S \cong H$ and $N_{G_SS}(G_S)=N \cong \mathbb{Z}_4$. With these information, we observe that $\sigma(S)_1=\{1,w\}$. We show that $\sigma(S)$ is not a congruence by showing that it does not satisfy the condition $(2)$ of \cite[Theorem 2.7, p. 2686]{rpsc}.

Take $x=w,y=u^{-1}$ and $z=w^2u$. Note that $y \circ z=y$, $(y \circ z)^{\prime}=z$ and $y \circ (x \circ z)=xv^{-1}u$, where $\circ$ is the underlying binary opearation of $S$ and the convention for the product in this case is given as $(rs)(x)=r(s(x))$ for $r,s \in Sym(6)$ and $x \in \{1,\cdots,6\}$. By equation $(C6)$ of \cite[Definition 2.1, p. 71]{rltr}, we observe that $(f(z,y))^{-1}(xv^{-1}u)=xvu^{-1}$. Finally, $(xvu^{-1}) \circ z = xv^{-1}u^{-1} \notin \sigma(S)_1$. This fails the condition $(2)$ of \cite[Theorem 2.7, p. 2686]{rpsc}, showing that $\sigma(S)$ is not a congruence for the right loop $S$.
\end{example}

\begin{proposition}\label{6c3sp3} For a right loop $S$, the center $\mathcal{Z}(S) \subseteq \sigma(S)_1=S\cap N_{G_SS}(G_S)$.
\end{proposition}
\begin{proof} Let $\zeta$ be the congruence on $S$ determined by $\mathcal{Z}(S)$. Let $\gamma=S \times S$ and $\gamma$ centralize $\zeta$ by a centering congruence $(\gamma|\zeta)$. Let $x \in \mathcal{Z}(S)$ and $y,z \in S$. Since $(\gamma|\zeta)$ is reflexive, $(x,1)~(\gamma|\zeta)~(x,1)$. By condition $(iii)$ of Definition \ref{6c2sd1}, $(1,1)~(\gamma|\zeta)~(y,y)$, $(1,1)~(\gamma|\zeta)~(z,z)$ and $(1,1)~(\gamma|\zeta)~(y \circ z,y \circ z)$. Since $(\gamma|\zeta)$ is a congruence, $(x,1)=(((x,1)\circ(1,1))\circ(1,1))/(1,1)~(\gamma|\zeta)~(((x,1)\circ(y,y))\circ(z,z))/(y \circ z,y\circ z)=(((x\circ y)\circ z)/(y\circ z),1)$.

By the condition $(ii)$ and $(iv)$ of Definition \ref{6c2sd1}, $x=((x\circ y)\circ z)/(y\circ z)$. This implies that $(x \circ y)\circ z=x \circ (y\circ z)$. Thus, $\mathcal{Z}(S) \subseteq \sigma(S)_1$. 
\end{proof}

\begin{proposition}\label{6c3sp1} Let $(S,\circ)$ be a right loop. Let $\beta$ and $\gamma$ be a congruences on $S$. Assume that $\beta$ is centralized by $\gamma$ by means of $(\gamma|\beta)$. Then 
\begin{enumerate}
\item[(A)] $c \circ b^{\prime}=b^{\prime} \circ c$ for all $b \in \beta_1$ and $c \in \gamma_1$, where $u^{\prime}$ is left inverse of $u \in S$.

\item[(B)] if $\gamma=S \times S$, then $b \circ c=c \circ b$ for all $b \in \beta_1$ and $c \in S$.
\end{enumerate}
\end{proposition}
\begin{proof} $(A)$ Let $B=\beta_1$ and $C=\gamma_1$. Let $b \in B$ and $c \in C$. Since $(\gamma|\beta)$ is reflexive, 
 \begin{enumerate} \item[(i)] $(b,1)~(\gamma|\beta)~(b,1)$ and
                   \item[(ii)] $(b^{\prime},1)~(\gamma|\beta)~(b^{\prime},1)$, where $u^{\prime}$ denote the left inverse of $u \in S$.
 \end{enumerate} By $(iii)$ of Definition \ref{6c2sd1}, \begin{enumerate} \item[(iii)] $(1,1)~(\gamma|\beta)~(c,c)$. \end{enumerate}
 Since $(\gamma|\beta)$ is congruence, by equations $(iii), (ii)$ and $(i)$ we have \begin{equation}\label{3ce1} (1,1)~(\gamma|\beta)~((c\circ b^{\prime}) \circ b,c)\end{equation}
 Since $(\gamma|\beta)$ is symmetric and transitive, $(iii)$ and (\ref{3ce1}) imply that \begin{equation}\label{3ce2} (c,c)~(\gamma|\beta)~((c\circ b^{\prime}) \circ b,c)\end{equation}
 The condition $(ii)$ and $(iv)$ of Definition \ref{6c2sd1} imply that \begin{equation}\label{3ce3} (c \circ b^{\prime})\circ b=c \end{equation}
 Similarly, by equations $(ii),(iii), (i)$ and by arguing as for (\ref{3ce2}) and (\ref{3ce3}) we have \begin{equation}\label{3ce4} (b^{\prime} \circ c) \circ b=c\end{equation}
 By (\ref{3ce3}) and (\ref{3ce4}), we have \begin{equation}\label{3ce5} c \circ b^{\prime}=b^{\prime} \circ c\end{equation}
 
\noindent $(B)$ Let $\gamma=S \times S$. Since the congruence $\zeta$ determined by $\mathcal{Z}(S)$ is the maximal congruence centralized by $S \times S$, $\beta_1 \subseteq \mathcal{Z}(S)$. Also since $(\sigma(S)_1,\circ)$ is a group, by Proposition \ref{6c3sp3}, $(\beta_1,\circ)$ is a group. Hence, by $(A)$ \[b \circ c=c \circ b ~ \text{for~all}~ b \in \beta_1, c \in S.\] 
 
\end{proof}
\begin{remark} By Propositions \ref{6c3sp3} and \ref{6c3sp1}, the center $\mathcal{Z}$ of a right loop $S$ is an abelian group.
\end{remark}

One observes that if the normalizer $N_{G_SS}(G_S)$ is normal in $G_SS$, then $\sigma(S)$ becomes a congruence on $S$ (for $\sigma(S)_1$ is the kernel of the homomorphism from $S$ to $G_SS/N_{G_SS}(G_S)$ given by $x \mapsto N_{G_SS}(G_S)x$). For a loop, $\sigma$ concides with the center congruence (see \cite[Proposition 3.15, p. 83]{smth1}) but it need not concide for a right loop (when $\sigma(S)$ is a congruence). This is shown in following example:

\begin{example}\label{6c3se2} Let $G=Alt(4)$ and $H=\{I, (1,2)(3,4)\}$. Consider the ordered set $S=\{I,(1,3)(2,4),(1,2,3),(1,3,2),(2,3,4),(1,3,4)\}$. Then $S$ is an NRT of $H$ in $G$. Let $x_i \in S$ be the element placed at $i^{th}$ place. Then for the convenience of notation, we identify $x_i$ by $\underline{i}$. Now the multiplication table of $S$ is given as follows:     
\begin{center}
\begin{tabular}{c|cccccc} 
$\circ$  & \underline{1} & \underline{2} & \underline{3} & \underline{4} & \underline{5} & \underline{6} \\ 
\hline 
\underline{1} & \underline{1} & \underline{2} & \underline{3} & \underline{4} & \underline{5} & \underline{6} \\ 
\underline{2} & \underline{2} & \underline{4} & \underline{1} & \underline{1} & \underline{3} & \underline{2} \\
\underline{3} & \underline{3} & \underline{6} & \underline{2} & \underline{5} & \underline{6} & \underline{4} \\
\underline{4} & \underline{4} & \underline{1} & \underline{5} & \underline{2} & \underline{1} & \underline{3} \\
\underline{5} & \underline{5} & \underline{3} & \underline{6} & \underline{6} & \underline{4} & \underline{5} \\
\underline{6} & \underline{6} & \underline{5} & \underline{4} & \underline{3} & \underline{2} & \underline{1} \\
\end{tabular}
\end{center}
One can observe that $G_SS\cong G$, the alternating group of degree $4$, $\left|G_S\right|=2$ and $N_{G_SS}(G_S)\cong \mathbb{Z}_2 \times \mathbb{Z}_2 \trianglelefteq G_SS$. Then $\sigma(S)_1=\{\underline{1},\underline{6}\}$ is an invariant right subloop. If $\mathcal{Z}(S)$ is the center of $S$, then by Proposition \ref{6c3sp3} the order $\left|\mathcal{Z}(S)\right|=1$ or $\left|\mathcal{Z}(S)\right|=2$. If possible, assume that $\left|\mathcal{Z}(S)\right|=2$. Then $\underline{6} \in \mathcal{Z}(S)$ (Proposition \ref{6c3sp3}).  But $\underline{6} \circ \underline{2} \neq  \underline{2} \circ \underline{6}$. This is a contradiction to the Proposition \ref{6c3sp1}. Thus, $\left|\mathcal{Z}(S)\right|=1$. This shows that in general the center congruence need not concide with $\sigma(S)$.   
\end{example}

Let $T$ be an invariant right subloop of finite right loop $S$. By \cite[Theorem 2.7, p. 2686]{rpsc}, $\delta=\{(x \circ y, y)| x \in T, y \in S\}$ is the congruence determined by $T$. Then the equivalence class $\delta_z=T \circ z$ is in bijection with $T$. Which means that the order of $T$ divides the order of $S$.
In the following example, we calculate the centers of all right loops of order upto $5$:

\begin{example}\label{6c3se3} Since the center of a right loop is an invariant right subloop, all the right loops of order $3$ and $5$ which are not group has trivial center and a right loop of order $2$ is a group.
 
Let $(S,\circ)$ be a right loop of order $4$. Since $Core_{G_SS}(G_S)=\{1\}$, $G_SS$ is isomorophic to a subgroup of $Sym(4)$, the symmetric group of degree $4$. By the structure of $Sym(4)$, we have following choices:

\noindent $(i)$ $G_SS\cong S \cong \mathbb{Z}_4$, $\left|G_S\right|=1$, \qquad $(ii)$ $G_SS\cong S \cong \mathbb{Z}_2 \times \mathbb{Z}_2$, $\left|G_S\right|=1$,\\
$(iii)$ $G_SS\cong Sym(4)$, $G_S\cong Sym(3)$,\\
$(iv)$ $G_SS\cong Alt(4)$, $G_S\cong Alt(3)$ and\\
$(v)$ $G_SS\cong D_8$, $\left|G_S\right|=2$  

For the cases $(i)$ and $(ii)$, we have $\mathcal{Z}(S)\cong S$. The normalizer $N_{G_SS}(G_S)=G_S$ for the cases $(iii)$ and $(iv)$. This means that $\sigma(S)_1=\{1\}$. Therefore, by Proposition \ref{6c3sp3} $\mathcal{Z}(S)=\{1\}$ for the cases $(iii)$ and $(iv)$. Finally, consider the case $(v)$. Let $G_S=\{1,h\}$. Note that $N_{G_SS}(G_S) \cong \mathbb{Z}_2 \times \mathbb{Z}_2 \trianglelefteq G_SS$. This means that $\left|\sigma(S)_1\right|=2$. Let $S=\{1,x,y,z\}$ and $\sigma(S)_1=\{1,x\}$. Observe that $R_x^2=I_S$. Write $R_x=(1,x)(y,z)$ as a product of transpositions. If $x \in Z(G_SS)$, then $x \circ u=u \circ x$ for all $u \in S$. Assume that $hx \in Z(G_SS)$. If $x \circ y=x$, then either $xy=x$ or $xy=hx$, where juxtaposition is the binary operation in $G_SS$. But then either $y=1$ or $y=h$. This is a contradiction. By similar arguments, $x \circ y=y$ will give a contradiction. Hence $x \circ y = z$. Similarly $x \circ z=y$. Therefore, we are finally left with following table: 
\begin{center}
\begin{tabular}{c|cccc} 
$\circ$  & 1 & x & y & z  \\ 
\hline 
1 & 1 & x & y & z \\ 
x & x & 1 & z & y \\
y & y & z & $\star$ & $\star$ \\
z & z & y & $\star$ & $\star$ \\
\end{tabular}
\end{center}
Since $S$ is a right loop, $u \circ v \in \{1,x\}$, where $u,v \in \{y,z\}$. If $y \circ y=1$ and $y \circ z=x$, then $S=\mathcal{Z}(S) \cong \mathbb{Z}_2 \times \mathbb{Z}_2$. If $y \circ y=x$ and $y \circ z=1$, then $S=\mathcal{Z}(S) \cong \mathbb{Z}_4$. Therefore, either $y \circ y=1$ and $y \circ z=1$ or $y \circ y=x$ and $y \circ z=x$. But, both the right loops are isomorphic with an isomorphism $p$ defined by $p(1)=1$, $p(x)=x$, $p(y)=z$ and $p(z)=y$. Assume that $y \circ y=1$ and $y \circ z=1$. Now, we will show that $\{1,x\}$ is the center of $S$. 

Let $\beta=\{(1,1),(x,x),(y,y),(z,z),(1,x),(x,1),(y,z),(z,y)\}$. Clearly, $\beta$ is a congruence on $S$. Let $\gamma=S \times S$. Let $\mathcal{X}=\{((x,x),(y,y))|(x,y) \in \gamma\}$. Consider $(\gamma|\beta)=\mathcal{X} \cup V \cup V^{-1}$, where $V=\{\left((1,x),(x,1)\right),\left((1,x),(y,z)\right),\\ \left((1,x),(z,y)\right),\left((x,1),(y,z)\right),\left((x,1),(z,y)\right),\left((y,z),(z,y)\right)\}$ . One can check that $(\gamma|\beta)$ is a right subloop of $\beta \times \beta$. By Proposition \ref{6c2sp1}, $(\gamma|\beta)$ is a congruence on $\beta$. One can also observe that $\gamma$ centralizes $\beta$ by means of $(\gamma|\beta)$. This shows that $\left|\mathcal{Z}(S)\right|\neq 1$. Since $\mathcal{Z}(S) \subseteq \sigma(S)_1$ (see Proposition \ref{6c3sp3}) and $\left|\sigma(1,x)\right|=2$, $\mathcal{Z}(S)=\{1,x\}$.

\end{example}

In Example \ref{6c3se2}, we have seen that $\sigma(S)$ need not concide with the center congruence on $S$. Theorem \ref{6c3st1} gives a necessary and a sufficient condition when $\sigma(S)$ concides with the center congruence. To prove the theorem, we need a lemma.
\begin{lemma}\label{6c3sl1} Let $(S,\circ)$ be a right loop. For any $x \in S$, let $x^{\prime}$ denote the left inverse of $x$ in $S$. Then
\begin{enumerate}
\item[(i)] if $x \in \sigma(S)_1$, then $(x \circ y)/z=x \circ (y/z)$ for all $y,z \in S$
\item[(ii)] if $x \in \sigma(S)_1$, then $x/y=x \circ y^{\prime}$ for all $y \in S$
\item[(iii)] if $\sigma(S)_1=Z(G_SS)$ and $x \in \sigma(S)_1$, then $x \circ (y \circ z)=(y \circ x) \circ z$ for all $y,z \in S$
\item[(iv)] if $\sigma(S)_1=Z(G_SS)$ and $y \circ (x \circ z)=(y \circ x) \circ z $ for all $x \in \sigma(S)_1$ and $y,z \in S$, then $y/(x \circ z)=x^{\prime} \circ (y/z)$.
\end{enumerate}
\end{lemma}
\begin{proof} $(i)$ Let $x \in \sigma(S)_1$. Then for any $y,z \in S$, \[x \circ y=x \circ (y/z \circ z)=(x \circ y/z) \circ z.\] This implies that $(x \circ y)/z=x \circ (y/z)$.
 
\noindent $(ii)$ Let $x \in \sigma(S)_1$. Then $x \circ (y \circ z)=(x \circ y) \circ z$ for all $y,z \in S$. Let $(S,G_S,\sigma^{\prime},f)$ be associated $c$-groupoid (see \cite[Definition 2.1, p. 71]{rltr}). Then $x \theta f(y,z)=x$ for all $y,z \in S$, where $\theta$ is the right action of $G_S$ on $S$. Note that $x/y=x \theta f(y^{\prime},y)^{-1} \circ y^{\prime}$. This implies that $x/y=x \circ y^{\prime}$.

\noindent $(iii)$ Assume that $\sigma(S)_1=Z(G_SS)$. Let $x \in \sigma(S)_1$ and $y,z \in S \subseteq G_SS$. Then \begin{equation} \label{6c3sl1e1} yxz=xyz \end{equation}
First consider the L.H.S. of (\ref{6c3sl1e1}).
\begin{equation} \label{6c3sl1e2}yxz=f(y,x)(y \circ x)z= f(y,x)f(y\circ x,z) (y \circ x) \circ z \end{equation}
Now, consider the R.H.S. of (\ref{6c3sl1e1}) 
\begin{equation} \label{6c3sl1e3}xyz=xf(y,z)(y \circ z)= f(y,z)x (y \circ z)=f(y,z)f(x,y \circ z) x \circ (y \circ z)  \end{equation}
By (\ref{6c3sl1e2}) and (\ref{6c3sl1e3}) and the uniqueness of expression, we have \[x \circ (y \circ z)= (y \circ x) \circ z.\]

\noindent $(iv)$ Assume that $\sigma(S)_1=Z(G_SS)$. Let $x \in \sigma(S)_1$ and $y \in S \subseteq G_SS$. Then $xy=yx$. This implies that $f(x,y)x\circ y =f(y,x)y \circ x$. By the uniqueness of expression 
\begin{equation} \label{6c3sl1e4}f(x,y)= f(y,x)
\end{equation}
and 
\begin{equation} \label{6c3sl1e5} x \circ y= y \circ x \end{equation}
Also, we observe that \\ 
$(y \theta f(z^{\prime},z)^{-1} \circ z^{\prime}) \circ z=y \circ (z^{\prime} \circ z)=y=y \circ ((x \circ z)^{\prime} \circ (x \circ z))$ \\
$=(y \theta f((x \circ z)^{\prime},(x \circ z))^{-1} \circ (x \circ z)^{\prime}) \circ (x \circ z)$\\
$=((y \theta f((x \circ z)^{\prime},(x \circ z))^{-1} \circ (x \circ z)^{\prime}) \circ x) \circ z ~ \text{(from~ the~assumption)}$\\
$=(x \circ (y \theta f((x \circ z)^{\prime},(x \circ z))^{-1} \circ (x \circ z)^{\prime})) \circ z$ (by (\ref{6c3sl1e5}))\\
$=x \circ ((y \theta f((x \circ z)^{\prime},(x \circ z))^{-1} \circ (x \circ z)^{\prime}) \circ z)$ (for $x \in \sigma(S)_1$)

This gives \\  $(y/z) \circ z  =  x \circ ((y/(x \circ z)) \circ z)$ (for $u/v=(u \theta f(v^{\prime},v) \circ v^{\prime}) \circ v$ for all $u,v \in S$)

\qquad ~ $\Leftrightarrow x^{\prime} \circ ((y/z) \circ z)$ 
   $=(y/(x \circ z)) \circ z$ (for $x \in \sigma(S)_1$) 

\qquad ~   $\Leftrightarrow (x^{\prime} \circ (y/z)) \circ z$ 
  $=(y/(x \circ z)) \circ z$ (for $x^{\prime} \in \sigma(S)_1$).  Hence \[x^{\prime} \circ (y/z)=y/(x \circ z)\]   
\end{proof}
\begin{theorem}\label{6c3st1} Let $(S,\circ)$ be a right loop. Assume that $\sigma(S)$ is a congruence on $S$. Then 
\begin{enumerate}
 \item[(A)] if $ \sigma(S)_1=Z(G_SS)$ and $(v \circ u) \circ w=v \circ (u \circ w)$ for all $u \in \sigma(S)_1$ and $v,w \in S$, then $\sigma(S)_1=\mathcal{Z}(S)$
 \item[(B)] if $\sigma(S)_1=\mathcal{Z}(S)$, then $ \sigma(S)_1 \subseteq Z(G_SS)$ and $N_{G_SS}(G_S)=G_S Z(G_SS)$.
\end{enumerate}
\end{theorem}
\begin{proof} $(A)$ Assume that $\sigma(S)_1=Z(G_SS)$. Since a congruence on $S$ is uniquely determined by its equivalence class at identity element (see \cite[Theorem 2.7, p. 2686]{rpsc}), $\sigma(S)=\{(a \circ x,x)|a \in \sigma(S)_1=Z(G_SS), x \in S\}$. 

Define a relation $\delta$ on $\sigma(S)$ by $\left((x,y),(u,v)\right) \in \delta$ if $x/y=u/v$. We will first prove that $\delta$ is a congruence on $\sigma(S)$. It is easy to see that $\delta$ is an equivalence relation on $\sigma(S)$. Thus, we need to check that it is a right subloop of $\sigma(S) \times \sigma(S)$. 

Note that $((u \circ x,x),(v \circ y,y)) \in \delta$$\Leftrightarrow u=v$. Let $\left((u \circ x_1,x_1),(u \circ t_1,t_1)\right)$ and $\left((v\circ x_2,x_2),(v \circ t_2,t_2)\right)$ be in $\delta$, where $u,v \in \sigma(S)_1$ and $x_1,x_2,t_1,t_2 \in S$. We first check that $(((u \circ x_1)/(v \circ x_2),x_1/x_2),((u \circ t_1)/(v \circ t_2),t_1/t_2)) \in \delta$. For this first observe that \\
$((u \circ x_1)/(v \circ x_2))/(x_1/x_2)$\\
$=(u \circ (x_1/ (v \circ x_2)))/(x_1/x_2)$ (by Lemma \ref{6c3sl1} $(i)$)\\
$=(u \circ (v^{\prime} \circ (x_1/x_2)))/(x_1/x_2)$ (by Lemma \ref{6c3sl1} $(iv)$)\\$=((u \circ v^{\prime}) \circ (x_1/x_2)) / (x_1/x_2)$ (for $u \in \sigma(S)_1$)\\=$u \circ v^{\prime}$

\noindent Similarly $(u \circ t_1/v \circ t_2)/(t_1/t_2)=(u \circ v^{\prime})$. Thus $(((u \circ x_1)/(v \circ x_2),x_1/x_2),((u \circ t_1)/(v \circ t_2),t_1/t_2)) \in \delta$. 

We now check that $(((u \circ x_1) \circ (v \circ x_2),x_1 \circ x_2),((u \circ t_1) \circ (v \circ t_2),t_1 \circ t_2)) \in \delta$. For this observe that \\
$((u \circ x_1) \circ (v \circ x_2))/(x_1/x_2)$\\
$=(((u \circ x_1)\circ v)\circ x_2)/(x_1/x_2)$ (by the assumption)\\
$=((u \circ (x_1 \circ v)) \circ x_2)/(x_1/x_2)$ (for $u \in \sigma(S)_1$)\\ $= ((u \circ (v \circ x_1)) \circ x_2)/ (x_1/x_2)$ (by \ref{6c3sl1e5} of the proof of Lemma \ref{6c3sl1})\\=$(((u \circ v) \circ x_1)\circ x_2)/ (x_1/x_2)$ (for $u \in \sigma(S)_1$)\\=$((u \circ v)\circ (x_1 \circ x_2))/ (x_1/x_2)$ (for $u \circ v \in \sigma(S)_1$)\\= $u \circ v$ 

\noindent Similarly $((u \circ t_1)\circ (v \circ t_2))/(t_1/t_2)=(u \circ v)$. Thus $(((u \circ x_1)\circ (v \circ x_2),x_1/x_2),((u \circ t_1) \circ (v \circ t_2),t_1/t_2)) \in \delta$. 
Hence, $\delta$ is a congruence on $\sigma(S)$. 

Let $\gamma=S \times S$. We now show that $\gamma$ centralizes $\sigma(S)$ by means of $\delta$. For this, we use Proposition \ref{6c2sp2}. Let $x \in S$. Then $\left((x,x),(x,y)\right) \in \delta$ $\Rightarrow 1=x/x=x/y\Rightarrow x=y$. This shows that the condition $(ii)$ of Proposition \ref{6c2sp2} is satisfied. One can trivially observe that the condition $(i)$ of Proposition \ref{6c2sp2} is also satisfied. Hence, by Proposition \ref{6c2sp2}, $\gamma$ centralizes $\sigma(S)$ by means of $\delta$. Since $\zeta(S)$ is the maximal central congruence centralized by $S \times S$, $\sigma(S) \subseteq \zeta(S)$, Thus, by Proposition \ref{6c3sp3} $\sigma(S)_1=\mathcal{Z}(S)$. 

\noindent $(B)$ Assume that $\sigma(S)_1=\mathcal{Z}(S)$. Let $x \in \sigma(S)_1$ and $y,z \in S$. By Proposition \ref{6c3sp2}, $x \in \sigma(S)_1$ is a left non-singular element of $S$. Let $\tilde{S}=\{R_u|u \in S\} $. Let $\langle \tilde{S} \rangle$ be a subgroup of $Sym(S)$ generated by $\tilde{S}$. Then $G_S$ is the stabilizer of $1$ in $\langle \tilde{S} \rangle$. Also $\tilde{S}$ is an NRT of $G_S$ in $\langle \tilde{S} \rangle$ and the map $u\mapsto R_u$ is an isomorphism of right loops. Thus, identifying NRTs $S$ and $\tilde{S}$, we have $\langle \tilde{S} \rangle=G_SS$. 

Since $\sigma(S)_1=\mathcal{Z}(S)$, by Proposition \ref{6c3sp1}$(B)$ if $x \in \sigma(S)_1$, then $L_x=R_x \in \langle \tilde{S} \rangle$. Let $x \in \sigma(S)_1$ and $y,z \in S$. Then, we have $(x \circ y) \circ z=x \circ (y \circ z)$. This implies that $R_z(L_x(y))=L_x((R_z(y))$, that is $R_xR_z=R_zR_x$ for all $z \in S$. Therefore, $R_x \in Z(\langle \tilde{S} \rangle)$, that is $x \in Z(G_SS)$.

Next, since $N_{G_SS}(G_S) \cap S=\sigma(S)_1$, $N_{G_SS}=G_S\sigma(S)_1 \subseteq G_SZ(G_SS)$. Obiously $G_SZ(G_SS) \subseteq N_{G_SS}(G_S)$. Thus $N_{G_SS}(G_S)=G_S Z(G_SS)$.   
\end{proof} 

\section{Quasiprimitivity and Simple Right Loop}\label{6c4s}

A permutation group $G$ on a set $\Omega$ is said to be \textit{quasiprimitive} \index{quasiprimitive} on $\Omega$ if each of its non-trivial normal subgroups is transitive on $\Omega$ (see \cite[p. 227]{ceprg}). A classification of quasiprimitive permutation group is obtained in \cite[Theorem 1, p. 227]{ceprg}. It is obtained in \cite[Theorem 2. p. 474]{phsmth} that a loop is simple if and only if the group generated by left translations and right translations of the loop is quasiprimitive on the loop. In this section, we will give an example of a right loop $S$ which is simple but the group $G_SS$ is not quasiprimitive on $S$. This example corrects \cite[Theorem 4, p. 474]{phsmth} and shows that it is one directional only.
\begin{example}\label{6c4se1} Let $G=D_{18}=<x, y: x^2=y^9=1, xyx=y^8>$ be the group of order 18. Let $H=\left\{1,x\right\}$ and $S=\left\{1, xy,..., xy^8\right\}$. 

The group $G_SS=\langle R_{xy^i}: 0 \leq i \leq 8 \rangle \cong D_{18}$. Let $\theta=R_{xy^2}R_{xy}$. Then $\theta$  is an element of order 9 in $G_SS$. Consider the normal subgroup $N=<\theta^3>=\left\{1, \theta^3, \theta^6 \right\}$ of $G_SS$ of order $3$. Obviously the right action of $N$ on $S$ is not transitive. The partition defined by the action of $N$ on $S$ is given by \[\left\{O_1=\left\{1, xy^3, xy^6 \right\}, O_2=\left\{xy, xy^4, xy^7 \right\}, O_3=\left\{xy^2, xy^5, xy^8 \right\} \right\}.\]

Let $S/N=\left\{O_1, O_2, O_3 \right\}$. Let $\ast$ be the binary relation on $S/N$ defined by $uN \ast vN=(u \circ v)N$. Then $1N=xy^3N$, however $xyN \ast 1N= (xy \circ 1)N=xyN$ and $xyN \ast xy^3N=(xy \circ xy^3)N=xy^2N \neq xyN $. Hence $\ast$ is not a binary operation on $S/N$ as claimed in the first paragraph of the proof of \cite[Theorem 4. p, 474]{phsmth}.

In fact, the right loop $S$ is simple. For, if $R$ is a nontrivial proper congruence on $S$, then the kernel of the epimorphism $\phi : G_SS \rightarrow G_{S/R}S/R$ induced by the quotient homomorphism $\nu : S \rightarrow S/R$ contains a nontrivial element $R_{xy^i}$ (for some $i, 1 \leq i \leq 8$). But, then $ker \phi $ contains all $R_{xy^k}$    for all $k$, $( 1 \leq k \leq 8)$ (since all $R_{xy^k}$'s are conjugate in $G_SS \cong D_{18}$ ). Hence, $ker \phi =G$, a contradiction. 

\end{example}

\end{document}